\documentclass[12pt]{amsart}

\usepackage[a4paper,left=2.2cm,right=2.2cm,top=2.5cm,bottom=2.5cm]{geometry}

\usepackage{amssymb, amsmath, amsthm, color, marginnote, hyperref, cancel ,amscd, amsfonts,mathtools}
\usepackage[nobysame]{amsrefs}

\usepackage{enumitem}
\usepackage[all]{xy}
\usepackage{hyperref}
\usepackage[T1]{fontenc} 
\usepackage{mathrsfs}
\usepackage{wasysym}
\usepackage{comment}

\usepackage{multicol}
\usepackage{rotating}
\usepackage{graphicx,fancyhdr}
\usepackage{multirow}
\usepackage[normalem]{ulem}

\allowdisplaybreaks 

\hypersetup{colorlinks=true,citecolor=cyan,linkcolor=blue,linktocpage=true}

\numberwithin{equation}{section}

\newtheorem{theorem}{Theorem}[section]

\newtheorem{lemma}[theorem]{Lemma}

\newtheorem{corollary}[theorem]{Corollary}
\newtheorem{proposition}[theorem]{Proposition}

\theoremstyle{definition}

\theoremstyle{remark}

\newcommand{\thmref}[1]{Theorem~\ref{#1}}

\newcommand{\lemref}[1]{Lemma~\ref{#1}}
\newcommand{\propref}[1]{Proposition~\ref{#1}}
\newcommand{\corref}[1]{Corollary~\ref{#1}}

\begin{document}

\title[Fomin-Kirillov algebras]{On the Fomin-Kirillov algebras}

\author{Yang Zhang}

\address{College of Science and Engineering, James Cook University, Townsville, QLD 4811, Australia}

\email{yang.zhang9\;\!@\;\!jcu.edu.au}

\begin{abstract}

The Fomin-Kirillov algebras $\mathcal{FK}_n$ are shown to be infinite-dimensional for $n\geq 6$. By computer calculations, we find an explicit nonzero primitive element of  degree 14 in the identity component of $\mathcal{FK}_6$. For every $n\geq 6$, the existence of such an  element implies  that  $\mathcal{FK}_n$ contains a polynomial subalgebra and is therefore infinite-dimensional, and that  the canonical homomorphism from $\mathcal{FK}_n$ onto its associated Nichols algebra is not an isomorphism.
\end{abstract}

\maketitle

\section{Introduction}

\subsection{Notation}
Let $G=\mathrm{Sym}_n$ be the symmetric group on $\{1,2,\dots,n\}$. The Fomin--Kirillov algebra
$\mathcal{FK}_n$ \cite{FK99} is the associative $\mathbb{C}$-algebra
generated by $x_{ij}=-x_{ji}$, $1\leq i\neq j\leq n$,
subject to
\begin{align*}
  x_{ij}^2&=0,\\
  x_{ij}x_{kl}&=x_{kl}x_{ij},\\
  x_{ij}x_{jk}+&x_{jk}x_{ki}+x_{ki}x_{ij}=0,
\end{align*}
where the indices in each relation are pairwise distinct.

The group $G$ acts on $\mathcal{FK}_n$ by $\sigma. x_{ij}=x_{\sigma(i)\sigma(j)}$. Besides its natural $\mathbb{Z}_{\geq0}$-grading with $\deg(x_{ij})=1$, the algebra has a $G$-grading
\[
  \mathcal{FK}_n=\bigoplus_{g\in G}(\mathcal{FK}_n)_g,
  \qquad \text{with}\quad  \deg_G(x_{ij})=(ij).
\]
These structures make $\mathcal{FK}_n$ a connected graded braided Hopf algebra in the category of Yetter--Drinfeld modules over $\mathbb{C}G$ \cite{MS00}. For $G$-homogeneous elements $u,v$, the braiding $c$ is given by 
\[
  c(u\otimes v)=(\deg_G u). v\otimes u,
\]
and the braided coproduct is determined by
\[
  \Delta(x_{ij})=x_{ij}\otimes1+1\otimes x_{ij}.
\]
An element $p\in\mathcal{FK}_n$ is called
\emph{primitive} if $\Delta(p)=p\otimes1+1\otimes p$. We denote by $(\mathcal{FK}_n)_{d,g}$ the subspace of $\mathbb{Z}_{\geq 0}$-degree $d$ and $G$-degree $g$.

It is well known that $\mathcal{FK}_n$ is finite-dimensional for $n\leq 5$ and that $\mathcal{FK}_n$ naturally embeds as a subalgebra of $\mathcal{FK}_{n+1}$ \cite{FK99}. The infinite-dimensionality of $\mathcal{FK}_n$ for $n\geq 6$ was an open problem.
To prove this for all $n\geq 6$, it suffices to consider $\mathcal{FK}_6$.

\subsection{Main results}

Our approach is to construct a nonzero homogeneous primitive element in the identity component $(\mathcal{FK}_6)_e$, where $e$ is the identity of $G=\mathrm{Sym}_6$.
If such an element $p$ exists, then $c(p\otimes p)=p\otimes p$, and hence for every integer $m\geq 1$,
\[
  \Delta(p^m)
  =\sum_{r=0}^m\binom{m}{r}
   p^r\otimes p^{m-r}.
\]
Over $\mathbb{C}$, this forces all powers of $p$ to be nonzero. Consequently, $\mathbb{C}[p]$ is a polynomial subalgebra of $\mathcal{FK}_6$, proving that $\mathcal{FK}_6$ is infinite-dimensional.

To construct such an element, we make use of the tensor product decomposition established by Fomin and Procesi~\cite{FP00}. For any integer $n\geq 2$,  let $\mathcal{C}_{n-1}$ be the subalgebra of
$\mathcal{FK}_n$ generated by 
\[x_i:=x_{in}, \qquad  1\leq i\leq n-1. \]
As shown in \textit{loc. cit.}, multiplication induces an isomorphism of graded vector spaces
\[
  \mathcal{C}_{n-1}\otimes\mathcal{FK}_{n-1}
  \xrightarrow{\ \cong\ }\mathcal{FK}_n,
  \qquad a\otimes b\mapsto ab.
\]

Motivated by this decomposition, we restrict our attention to $\mathcal{C}_5\subseteq\mathcal{FK}_6$ and seek a nonzero element
\[
  p\in(\mathcal{C}_5)_{d,e} \subseteq (\mathcal{FK}_6)_{d,e},\qquad d>1,
\]
which is primitive with respect to the braided coproduct of $\mathcal{FK}_6$. For a word $w=x_{i_1}\cdots x_{i_d}\in \mathcal C_5$, define its \emph{support} by
\[
  \operatorname{Supp}(w)
  =\{i\in\{1,\ldots,5\}: x_i\text{ occurs in }w\}.
\]
We say that $w$ has \emph{full support} if
$\operatorname{Supp}(w)=\{1,\ldots,5\}$.

We prove  that if a word $w\in (\mathcal{C}_5)_{d,e}$ is nonzero, then $d$ is even and 
\[|\operatorname{Supp}(w)|\leq \min\left\{\frac{d}{2}-1,\,5\right\}.\]
Moreover, every homogeneous primitive element in $\mathcal{C}_5$ of degree greater than $1$ is a linear combination of words with full support. Thus, a nonzero primitive element in $(\mathcal{C}_5)_{d,e}$  exists only if $d\geq 12$.

The subgroup $\mathrm{Sym}_5\subseteq\mathrm{Sym}_6$ fixing
$6$ acts on $\mathcal{C}_5$ by $\sigma. x_i=x_{\sigma(i)}$ for any $\sigma\in\mathrm{Sym}_5$,
and preserves each homogeneous component $(\mathcal{C}_5)_{d,e}$. We further restrict our attention  to the invariant subspaces $(\mathcal{C}_5)_{d,e}^{\mathrm{Sym}_5}$ for $d\geq12$.
By computer calculations,  there are no such  nonzero primitive elements in degree $12$, but one exists in degree $14$.

\begin{theorem}\label{thm: main1}
  The Fomin--Kirillov algebra $\mathcal{FK}_6$ contains a nonzero $\mathrm{Sym}_5$-invariant primitive element of degree $14$ in the identity component.
  Consequently, $\mathcal{FK}_6$ contains a polynomial subalgebra and is infinite-dimensional.
\end{theorem}

An explicit expression for the primitive element is given in Table~\ref{tab:primitive} as a linear combination of $\mathrm{Sym}_5$-orbit sums.

There is a further consequence concerning the relation between the Fomin--Kirillov algebra and its associated Nichols algebra. Let $V_n$ denote the Yetter--Drinfeld module spanned
by $x_{ij}, 1\leq i\neq j\leq n$, with the braiding $c$, and let $\mathcal{B}(V_n)$ be its Nichols algebra. There is a canonical surjective homomorphism of graded braided
Hopf algebras
\[
  \pi_n:\mathcal{FK}_n\longrightarrow\mathcal{B}(V_n).
\]
Since  all primitive elements of $\mathcal{B}(V_n)$ have degree 1,  any higher degree primitive element of $\mathcal{FK}_n$ belongs to $\ker\pi_n$. We thus obtain the following consequence.

\begin{theorem}\label{thm: main2}
  The canonical homomorphism
  \[
    \pi_n:\mathcal{FK}_n\longrightarrow\mathcal{B}(V_n)
  \]
  is not bijective for $n\geq 6$.
\end{theorem}

It is well known that  $\pi_n$ is an isomorphism for $n\leq 5$ \cites{Gra,GGI11,MS00}. \thmref{thm: main2} implies that $\mathcal{B}(V_n)$ is not a quadratic algebra for $n\geq 6$, disproving a conjecture of Milinski and Schneider \cite{MS00}.

Recently, Andruskiewitsch et al.~\cite{AAHLRV26} proved the infinite-dimensionality of the
Fomin--Kirillov algebras $\mathcal{FK}_n$ for $n\geq 6$ using methods from the theory of Nichols algebras.

Lastly, our interest in Fomin-Kirillov algebras stems from their connections with the Milnor fibres of reflection arrangements.  The Fomin--Kirillov algebra is, up to a cocycle twist \cite{Ven12}, isomorphic to a  braided Hopf algebra initially studied in ~\cites{MS00}\footnote{I would like to thank Giovanna Carnovale, Francesco Esposito, and Francesco Sala for sharing their insights.}, which has a finite-dimensional quotient known as the noncrossing algebra \cites{Zha23,LZ24}.  The latter is used to construct chain and cochain complexes computing the integral homology and cohomology of Milnor fibres and hyperplane complements associated with finite Coxeter groups, as well as those of the corresponding Artin groups. It would be interesting to investigate whether our results here have geometric or topological interpretations.

\vspace{1em}
\noindent {\bf Acknowledgement.} I am grateful to Gus Lehrer and Ruibin Zhang for many enlightening  discussions and  generous encouragement   over the years. I also thank Nicol\'as Andruskiewitsch for discussions  during his visit to the University of Queensland from May 25 to June 10, 2026. The main idea and computational results were developed independently  and were communicated to him by email on 20 July 2026. 

\vspace{1em}
\noindent {\bf AI.} ChatGPT was used to produce SageMath code and polish the writing.

\section{Low-degree analysis}

In this section, we prove that a nonzero homogeneous element in the identity component of $\mathcal{C}_5$ is primitive only if it is a linear combination of words with full support and has degree at least $12$.

\subsection{The identity and support}
Let $(\mathcal{C}_5)_e$ denote the identity component of
$\mathcal{C}_5$, and let $(\mathcal{C}_5)_{d,e}$ be its
homogeneous component of $\mathbb{Z}_{\geq 0}$-degree $d$.
Set
\[
t_i:=(i6)\in G=\mathrm{Sym}_6,\qquad 1\leq i\leq 5.
\]
Then $\mathrm{deg}_G(x_i)=t_i$ for each $i$. Since $t_it_j=e$ if and
only if $i=j$, and $x_i^2=0$, we have
$(\mathcal{C}_5)_{2,e}=\{0\}$.
Moreover, since each $t_i$ is an odd permutation, an element
of identity degree has even
$\mathbb{Z}_{\geq 0}$-degree. Hence, in what follows,
\begin{center}
\textit{we consider only even degrees $d\geq 4$.}
\end{center}
We abbreviate a word $w=x_{i_1}\cdots x_{i_d}$ as $i_1\cdots i_d$.

\begin{lemma}\label{lem: supp}
Let $w\in (\mathcal{C}_{5})_e$ be a nonzero word. Then every $i\in \mathrm{Supp}(w)$ occurs in $w$ at least twice.
\end{lemma}

\begin{proof}
Suppose for contradiction that some $i\in \mathrm{Supp}(w)$ occurs exactly once in $w$. Write $w=w_1iw_2$. Since $w\in (\mathcal{C}_{5})_e$, we have
\[
\alpha_1 t_i\alpha_2=e,
\]
where $\alpha_1 =\mathrm{deg}_G(w_1)$ and $\alpha_2=\mathrm{deg}_G(w_2)$. Because $i$ occurs only once in $w$, both $\alpha_1 $ and $\alpha_2$ are products of reflections $t_j=(j6)$ with $j\neq i$. Hence every reflection factor of $\alpha_1 $ and $\alpha_2$ fixes $i$, so both $\alpha_1 $ and $\alpha_2$ fix $i$. It follows from $\alpha_1 t_i\alpha_2=e$ that
\[
t_i=\alpha_1 ^{-1}\alpha_2^{-1}.
\]
So $t_i$ fixes $i$, since both $\alpha_1 ^{-1}$ and $\alpha_2^{-1}$ do. But $t_i=(i6)$ does not fix $i$, a contradiction.
\end{proof}

\begin{proposition}\label{prop: ds}
Let $w\in (\mathcal{C}_{5})_{d,e}$ be a nonzero word and set $s=|\mathrm{Supp}(w)|$. Then $d\geq 2s+2$.
\end{proposition}

\begin{proof}
By \lemref{lem: supp}, every $i\in \mathrm{Supp}(w)$ occurs in $w$ at least twice. Since $\mathrm{Supp}(w)$ contains $s$ distinct elements, it follows that $d\geq 2s$.

Suppose $d=2s$. Then every $i\in \mathrm{Supp}(w)$ occurs in $w$ exactly twice. We first observe that deleting the two occurrences of any fixed letter preserves the identity group degree. Indeed, write
$w=w_1aw_2aw_3$ for some $a\in \mathrm{Supp}(w)$, and let
$\alpha_k=\mathrm{deg}_G(w_k)$ for $k=1,2,3$. Since $a$ occurs exactly twice in $w$, none of the words $w_k$ contains $a$, and hence $\alpha_k(a)=a$. Taking group degrees gives
\[
e=\alpha_1t_a\alpha_2t_a\alpha_3.
\]
Applying both sides to $a$, we obtain
\[
a=\alpha_1t_a\alpha_2t_a\alpha_3(a)
 =\alpha_1t_a\alpha_2(6).
\]
Since $\alpha_1$ fixes $a$, it follows that $t_a\alpha_2(6)=a$, and hence $\alpha_2(6)=6$. Thus $\alpha_2$ fixes both $a$ and $6$, so $\alpha_2$ commutes with $t_a=(a6)$. Therefore, we obtain
\[
e=\alpha_1t_a\alpha_2t_a\alpha_3
 =\alpha_1\alpha_2\alpha_3.
\]
Thus, the word $w'=w_1w_2w_3$, obtained by deleting the two occurrences of $a$, lies in $(\mathcal{C}_{5})_{d-2,e}$.

We now show that the two occurrences of distinct letters cannot cross. Suppose for contradiction that there exist $a,b\in\mathrm{Supp}(w)$ such that
\[
w=w_1\,a\,w_2\,b\,w_3\,a\,w_4\,b\,w_5.
\]
By successively deleting the two occurrences of every letter other than $a$ and $b$, using the preceding observation at each step, we obtain the word $abab$ of identity group degree. Hence
\[
t_at_bt_at_b=e.
\]
However, $t_at_b=(a6)(b6)$ is a $3$-cycle for $a\neq b$, and therefore $(t_at_b)^2\neq e$, a contradiction.

Thus, the pairs of equal letters in $w$ are noncrossing. It follows that some pair must be adjacent. Indeed, choose a letter whose two occurrences have minimal distance. If there were a letter between them, then, since crossings are impossible, the other occurrence of that letter would also lie between them, contradicting the minimality of the chosen pair. Thus, $w$ contains a subword $aa$ for some $a\in\mathrm{Supp}(w)$. Since $a^2=0$ in $\mathcal{C}_{5}$, we have $w=0$, contradicting our assumption.

Therefore $d\neq 2s$. Since $d$ is even and $d\geq 2s$,
we have $d\geq 2s+2$.
\end{proof}

We have the following immediate consequence of \propref{prop: ds}.
\begin{corollary}\label{coro: supp}
For every nonzero word $w\in (\mathcal{C}_{5})_{d,e}$ with $d\geq 4$ even, we have 
\begin{equation*}\label{eq: sdine}
|\mathrm{Supp}(w)|\leq \min\left\{\frac{d}{2}-1,\,5\right\}.  
\end{equation*}
\end{corollary}

\subsection{Primitivity and support} We proceed next to establish a lower bound on the support size of primitive elements in $\mathcal{C}_{n-1}\subseteq \mathcal{FK}_n$. Here, primitivity is defined with respect to the braided coproduct of $\mathcal{FK}_{n}$. In general,  $\mathcal{C}_{n-1}$ is not closed under the braided coproduct. To understand the braided coproduct, we first introduce a decomposition of $\mathcal{FK}_n$ based on the indices occurring in monomials.

For a monomial $m=x_{i_1j_1}\cdots x_{i_dj_d}$ in the generators of $\mathcal{FK}_n$, define its
\emph{index support} by
\[
I(m):=\bigcup_{r=1}^{d}\{i_r,j_r\}.
\]
We also set $I(1)=\varnothing$.
Since $I(mm')=I(m)\cup I(m')$ and the defining relations are homogeneous with respect to index
support, $\mathcal{FK}_n$ admits a decomposition into vector spaces: 
\[
\mathcal{FK}_n =\bigoplus_{T\subseteq\{1,\ldots,n\}}\mathcal{FK}_n(T),
\]
where $\mathcal{FK}_n(T)$ is the subspace spanned by monomials with index support $T$. 

Now the following lemma is straightforward, but we include a proof for completeness.

\begin{lemma}\label{lem: support-coprod}
For every $T\subseteq\{1,\ldots,n\}$, the braided coproduct satisfies
\[
\Delta(\mathcal{FK}_n(T))\subseteq \bigoplus_{U\cup V=T}\mathcal{FK}_n(U)\otimes \mathcal{FK}_n(V).
\]
\end{lemma}

\begin{proof}
By linearity, it suffices to prove that for every monomial $m\in \mathcal{FK}_n(T)$, we have 
\[
\Delta(m)=\sum_{(m)}m_{(1)}\otimes m_{(2)}, 
\]
with $I(m_{(1)})\cup I(m_{(2)})=T$ for every summand. Here, we use Sweedler's notation. 

We proceed by induction on the $\mathbb{Z}_{\geq 0}$-degree of $m$. The assertion holds for $m=1$ and for the generators $x_{ij}$, since $\Delta(x_{ij})=x_{ij}\otimes 1+1\otimes x_{ij}$. Suppose the assertion holds for a monomial $m$. Since $\Delta$ is an algebra homomorphism into the
braided tensor product, we have
\[
\Delta(mx_{ij}) =\Delta(m)\Delta(x_{ij}) =\sum_{(m)}\left(m_{(1)}(g. x_{ij})\otimes m_{(2)}
+m_{(1)}\otimes m_{(2)}x_{ij}
\right),
\]
where $g=\deg_G(m_{(2)})$.

Since $g$ is a product of transpositions involving only indices in $I(m_{(2)})$, it preserves
$I(m_{(2)})$ and fixes every index outside this set. Thus,
\[
I(m_{(2)})\cup\{g(i),g(j)\} =I(m_{(2)})\cup\{i,j\},
\]
regardless of whether $i$ or $j$ belongs to $I(m_{(2)})$. Therefore, for the above  expansion of $\Delta(mx_{ij})$, 
\[\begin{aligned}
&I(m_{(1)}(g.x_{ij}))\cup I(m_{(2)}) =I(m_{(1)})\cup \{g(i),g(j)\}\cup I(m_{(2)}) =I(mx_{ij}),\\ 
&I(m_{(1)})\cup I(m_{(2)}x_{ij}) =I(m_{(1)})\cup I(m_{(2)})\cup\{i,j\} =I(mx_{ij}), 
\end{aligned}
\]
where we have used the induction hypothesis $I(m)= I(m_{(1)})\cup I(m_{(2)})$. The assertion holds.
\end{proof}

\begin{proposition}\label{prop: supp}
Let $n\geq 2$. If $p\in \mathcal{C}_{n-1}$ is a homogeneous primitive element of degree $d>1$, then $p$ is a linear combination of words $w$ satisfying
\[|\mathrm{Supp}(w)|\geq 5.\]
In particular, every such element in $\mathcal{C}_5$ is a linear combination of words with full support.
\end{proposition}

\begin{proof}
Write
\[
p=\sum_{S\subseteq\{1,\ldots,n-1\}}p_S,
\]
where $p_S$ is a linear combination of words $w$ with
$\mathrm{Supp}(w)=S$ and degree $d>1$. We claim that each $p_S$ is primitive. Indeed, since $p$ is primitive, we have
\[
0=\Delta(p)-p\otimes1-1\otimes p =\sum_{S\subseteq\{1,\ldots,n-1\}}
\left(\Delta(p_S)-p_S\otimes1-1\otimes p_S\right).
\]
By \lemref{lem: support-coprod}, for every nonempty $S\subseteq\{1,\ldots,n-1\}$,
\[
\Delta(p_S)-p_S\otimes1-1\otimes p_S
\in \bigoplus_{U\cup V=S\cup\{n\}}\mathcal{FK}_n(U)\otimes \mathcal{FK}_n(V).
\]
Since these subspaces are distinct direct summands of $\mathcal{FK}_n\otimes\mathcal{FK}_n$ for different $S$, the corresponding terms do not cancel.
Therefore, $\Delta(p_S)=p_S\otimes1+1\otimes p_S$ for every $S$, so each $p_S$ is primitive.

Now suppose that $p_S\neq0$ for some $S$ with $s=|S|\leq4$, and put $T=S\cup\{n\}$. Let $FK_T$ be the subalgebra of $FK_n$ generated by $x_{ij}$ with $i,j\in T$.
Then $p_S$ is a nonzero primitive element of degree $d>1$ in the braided Hopf subalgebra
$\mathcal{FK}_T\cong\mathcal{FK}_{s+1}$. Since $s+1\leq5$, this algebra is isomorphic
to its associated Nichols algebra \cite[Proposition~2.5]{GGI11}. However, the primitive space of a Nichols algebra is precisely its degree-$1$ component, a contradiction. Hence we have $p_S=0$ whenever $|S|\leq4$. This completes the proof. 
\end{proof}

\begin{corollary}\label{coro: dfull}
Let $p\in(\mathcal{C}_5)_{d,e}$ be a homogeneous primitive element of degree $d>1$.
Then $d\geq 12$, and $p$ is a linear combination of words with full support.
\end{corollary}

\begin{proof}
This follows by combining \corref{coro: supp} and \propref{prop: supp}
\end{proof}

\section{Higher-degree primitive elements}

By \corref{coro: dfull}, it suffices to consider primitive elements of even degree $d\geq 12$ in the subspace of $(\mathcal{C}_5)_{d,e}$ spanned by words with full support. We further restrict to $\mathrm{Sym}_5$-invariant subspaces. 

 Using SageMath, we find no nonzero primitive element with full support in $(\mathcal{C}_5)_{12,e}^{\mathrm{Sym}_5}$. In degree $14$, however, we obtain a nonzero primitive element $p\in(\mathcal{C}_5)_{14,e}^{\mathrm{Sym}_5}$, given explicitly by
\[
p=\sum_w c_w\,\Xi(w),
\qquad
\Xi(w)=\sum_{\sigma\in \mathrm{Sym}_5}\sigma(w).
\]
There are $45$ orbit representatives and $5400$
monomial terms. The $45$ pairs $(w, c_w)$ are given in Table \ref{tab:primitive}, where $w= x_{i_1}\cdots x_{i_d}$ is written as $i_1\cdots i_d$.

\begin{table}[htbp]
\centering

\begin{tabular}{c|r@{\qquad\qquad}c|r}
\hline
$w$ & $c_w$ & $w$ & $c_w$ \\
\hline
12321453545235 & -2  & 12321454325453 & 1  \\
12324541543235 & -3  & 12324541545323 & -1 \\
12341543545214 & 2   & 12342154532453 & -2 \\
12342515432415 & 2   & 12342545145324 & 2  \\
12342545153245 & -2  & 12343251543415 & 2  \\
12343254514534 & 4   & 12343254515345 & 4  \\
12343514342152 & 2   & 12343515434215 & -8 \\
12343521254342 & -3  & 12343521434525 & -2 \\
12343521435254 & -2  & 12343521452534 & 2  \\
12343521524345 & 10  & 12343521525434 & 2  \\
12343521543425 & -10 & 12345143545215 & -6 \\
12345214354525 & -4  & 12345215325434 & 2  \\
12345215453253 & 2   & 12345254515325 & 4  \\
12345315435212 & -2  & 12345325351434 & -6 \\
12345325451535 & -2  & 12345352135243 & -6 \\
12345352143524 & -6  & 12345352152435 & -14 \\
12345415453213 & 3   & 12345421254532 & 5  \\
12345421323545 & -9  & 12345421325453 & 4  \\
12345421354523 & -7  & 12345421532345 & 10 \\
12345421543235 & -10 & 12345421545323 & 10 \\
12345425451323 & 6   & 12345431545212 & -2 \\
12345432135453 & -3  & 12345435215245 & -10 \\
12345435421525 & 6   &                  &    \\
\hline
\end{tabular}
\vspace{1em}
\caption{Coefficients $c_w$ and orbit representatives $w$.}
\label{tab:primitive}
\end{table}

The existence of such a higher-degree primitive element implies Theorems \ref{thm: main1} and \ref{thm: main2}.

\end{document}